\documentclass[a4paper,twoside]{amsart}

\usepackage{amssymb,amsthm,amsmath,amscd,amsfonts,bbm,mathrsfs}
\usepackage{times}
\usepackage{hyperref}

\usepackage[all]{xy}
\SelectTips{cm}{}
\CompileMatrices

\theoremstyle{plain}
\newtheorem{Lemma}{Lemma}
\newtheorem{Thm}[Lemma]{Theorem}
\newtheorem{Prop}[Lemma]{Proposition}
\newtheorem{Cor}[Lemma]{Corollary}
\theoremstyle{remark}

\newtheorem{Example}[Lemma]{Example}
\newtheorem*{Rem*}{Remark}

\numberwithin{equation}{section}
\numberwithin{Lemma}{section}

\newcommand{\EEE}{\mathcal{E}}
\newcommand{\GGG}{\mathcal{G}}
\newcommand{\HHH}{\mathcal{H}}
\newcommand{\MMM}{\mathcal{M}}
\newcommand{\NNN}{\mathcal{N}}
\newcommand{\Fm}{\mathfrak{m}}
\newcommand{\Fn}{\mathfrak{n}}

\newcommand{\Fp}{\mathfrak{p}}
\newcommand{\Fq}{\mathfrak{q}}
\newcommand{\DD}{{\mathbb{D}}}
\newcommand{\EE}{{\mathbb{E}}}
\newcommand{\FF}{{\mathbb{F}}}
\newcommand{\GG}{{\mathbb{G}}}
\newcommand{\MM}{{\mathbb{M}}}
\newcommand{\QQ}{{\mathbb{Q}}}
\newcommand{\TT}{{\mathbb{T}}}
\newcommand{\ZZ}{{\mathbb{Z}}}

\DeclareMathOperator{\Def}{Def}
\DeclareMathOperator{\Ext}{Ext}
\DeclareMathOperator{\Gal}{Gal}
\DeclareMathOperator{\GL}{GL}
\DeclareMathOperator{\Hom}{Hom}
\DeclareMathOperator{\Lie}{Lie}
\DeclareMathOperator{\Lift}{Lift}
\DeclareMathOperator{\Quot}{Quot}
\DeclareMathOperator{\SL}{SL}
\DeclareMathOperator{\Spec}{Spec}
\DeclareMathOperator{\codim}{codim}
\DeclareMathOperator{\id}{id}
\DeclareMathOperator{\sep}{sep}

\newcommand{\hatR}{\widehat R}
\newcommand{\hatS}{\widehat S}

\newcommand{\hatOmega}{\widehat\Omega}
\newcommand{\hatGm}{\widehat\GG_m}

\title{Tate modules of universal $p$-divisible groups}
\author{Eike Lau}
\date{\today}

\begin{document}

\begin{abstract}
A $p$-divisible group over a complete local domain determines a
Galois representation on the Tate module of its generic fibre.
We determine the image of this representation
for the universal deformation in mixed characteristic 
of a bi-infinitesimal group 
and for the $p$-rank strata of the universal deformation 
in positive characteristic of an infinitesimal group. 
The method is a reduction 
to the known case of one-dimensional groups by a deformation
argument based on properties of the stratification by
Newton polygons.
\end{abstract}

\maketitle

\section{Introduction}

Let $G$ be a $p$-divisible group of dimension $d$ and height $c+d$
over an algebraically closed field $k$ of characteristic $p$.
Its universal deformation $\GGG$ is defined over a $W(k)$-algebra
$R$ isomorphic to an algebra of power series in $cd$ variables.
For every point $x\in\Spec R$ we have a natural Galois representation,
also referred to as local $p$-adic monodromy,
$$
\rho_x:\Gal(\bar x/x)\to\GL(T_p\GGG(\bar x))\cong\GL_{e(x)}(\ZZ_p)
$$
where $e(x)$ is the etale rank of the fibre $\GGG_x$.
Note that $e(x)=c+d$ if $x$ is a point of characteristic zero
and $e(x)\leq c$ if $x$ is a point of characteristic $p$.
Let $U_e$ be the locally closed subset of $\Spec R$ where $e(x)=e$.

\begin{Thm}
\label{Th1}
 $\:$

i) {}
If $G$ is bi-infinitesimal 
and $x$ is the generic point of\/ $\Spec R$
then the image of $\rho_x$ is the subgroup of all elements 
whose determinant is a $d$-th power.

ii) {}
If $G$ is infinitesimal 
and $x$ is a generic point of\/ $U_e$ for some 
$e\leq c$ then $\rho_x$ is surjective.
\end{Thm}

This is consistent with the general expectation 
that the monodromy of a universal family 
should be as large as possible, 
where the restriction in i) is caused by a well-known result of 
Raynaud \cite{Raynaud} saying that the determinant of $\rho_x$ 
is the $d$-th power of the cyclotomic character. 
The present article was motivated by recent work of
Tian \cite{Tian07} and Strauch \cite{Strauch07}, see below.
Instead of any attempt for a complete review of the
literature on $p$-adic monodromy of $p$-divisible groups 
and abelian varieties we refer the reader to 
\cite{Chai-Local}, \cite{Achter-Norman}, \cite{Chai-Methods} 
and the references given there.

If $G$ is one-dimensional, Theorem \ref{Th1} is proved in
\cite{Strauch07} using the theory of Drinfeld 
level structures. This result actually applies to one-dimensional 
formal modules over the ring of integers in a local field.
Previously, a number of cases were established by different methods.
Assertion i) for one-dimensional $p$-divisible groups is proved 
by Rosen and Zimmermann \cite{Rosen-Zimmermann,Zimmermann}.
Assertion ii) for the one-dimensional group of slope $1/2$ is a 
classical result of Igusa \cite{Igusa}, see \cite{Katz}, Theorem 4.3, 
while the case of slope $1/3$ is proved in \cite{Tian07} 
if $x$ is the generic point of $\Spec R/pR$.
For elementary $p$-divisible groups of arbitrary dimension 
the surjectivity of $\rho_x$ when $x$ is the generic point
of $\Spec R/pR$ has been conjectured in loc.\ cit. 

The present proof of Theorem \ref{Th1}
is a reduction to the one-dimensional case using 
the results of Oort and de Jong
\cite{Oort-Ann,Oort-PM,deJong-Oort} 
on the Newton stratification on $\Spec R/pR$.
More precisely, to prove ii) we pass to
the complete local ring (denoted $S$) of $R/pR$ 
at a generic point of a suitable Newton stratum
chosen so that $\GGG\otimes S$ 
is the extension of a one-dimensional
infinitesimal $p$-divisible group $\HHH$ and a group of
multiplicative type. 
Since the Tate module of the latter is trivial,
we only have to observe that
$\HHH$ over $S$ is necessarily the 
universal deformation of its special fibre,
see Lemma \ref{Le1} and Proposition \ref{Pr1}.

The proof of i) is more complicated because $p$-divisible 
groups of multiplicative type over a field of characteristic
zero have non-trivial Tate modules.
This leads us to consider different complete local rings 
of $R$ and their contributions to the image of the
Galois representation at the same time. 
Notably we need the following observation,
proved in Proposition \ref{Pr2}:
If $A$ denotes the complete local ring of $R$ at the prime
$pR$ and if $F'$ is an algebraic closure of the residue
field of $A$ then the set of ring homomorphisms $A\to W(F')$ 
lifting the given homomorphism $A\to F'$ is bijective to the 
set of deformations over $W(F')$ of the fibre $\GGG\otimes F'$.
As a consequence the contribution of $A$ to the Galois 
representation is sufficiently large, see Lemma \ref{Le3}.

Below we first explain the proof of Theorem \ref{Th1} 
in sections \ref{SecRed} and \ref{SecNull} and postpone 
the required Lemmas \ref{Le1} and \ref{Le3} until sections 
\ref{SecUni} and \ref{SecGen}.
They are straightforward applications of the
deformation theory of $p$-divisible groups 
developed in \cite{Messing} and \cite{Illusie}.
An alternative proof of Lemma \ref{Le1} in the case
where $G$ has $a$-number one is given in \cite{Tian08}.

It seems plausible that Theorem \ref{Th1} also holds
for special formal modules in the sense of Drinfeld.
This will be taken up in \cite{Lau-Strauch} along with
a study of the Newton stratification in that case.

\subsubsection*{Acknowledgement}
The author is grateful to M.~Strauch,
Y.~Tian, E.~Viehmann, T.~Wedhorn, and Th.~Zink for
interesting discussions on the subject of this article.

\section{Newton strata}

For reference let us recall the results 
on Newton strata we need. 
Let $\bar R=R/pR$.

The Newton polygon of a $p$-divisible group
$H$ is denoted $\NNN(H)$.
Newton polygons are normalised so that slope $0$ 
corresponds to etale groups and slope $1$ to groups
of multiplicative type.
The set of Newton polygons carries a partial order
such that $\beta\preceq\gamma$ if and only if
$\beta$ and $\gamma$ have the same endpoints 
and no point of $\beta$ lies strictly below $\gamma$.
The subset $V_\beta$ of $\Spec\bar R$ where
the Newton polygon of the universal deformation 
is $\preceq\beta$ is closed.
Denote by $V_\beta^\circ$ the open subset of $V_\beta$
where the polygon is equal to $\beta$
and let $\codim(\beta)$ be the number of lattice points 
that lie strictly below $\beta$ and on or above the unique 
ordinary Newton polygon with the same endpoints as $\beta$
(ordinary polygons are those whose slopes are all $0$ or $1$).

\begin{Thm}[\cite{Oort-PM}, Theorem 2.10]
\label{ThO}
The set $V_\beta$ is non-empty if and only if $\NNN(G)\preceq\beta$.
In that case all irreducible components of\/ $V_\beta$ 
have codimension $\codim(\beta)$ in $\Spec\bar R$, 
and consequently $V_\beta$ is the closure of\/ $V_\beta^\circ$.
Generically on $V_\beta$ the $a$-number is at most one.
\end{Thm}

At those points where the $a$-number is one
the strata are nested nicely:

\begin{Prop}
\label{PrO}
Let $x\in\Spec\bar R$ be given such that $a(\GGG_x)=1$.
Then every intersection $V_{\beta,x}=V_\beta\cap\Spec\bar R_x$ 
is regular and thus irreducible. 
If $V_{\beta,x}$ and $V_{\gamma,x}$ are non-empty, 
in other words $\NNN(\GGG_x)\preceq\beta$ 
and $\NNN(\GGG_x)\preceq\gamma$, 
then $V_{\beta,x}\subseteq V_{\gamma,x}$
if and only if $\beta\preceq\gamma$.
\end{Prop}

If $x$ is the maximal ideal of $\bar R$ 
this is \cite{Oort-Ann}, Theorem 3.2.
The general case can be reduced to this case
as is certainly well-known, but for completeness 
a proof is recalled in section \ref{SecDef} below.
We need the following supplement to Theorem \ref{ThO}.
Let us write $x\leq z$ if $x$ lies in the closure of $z$.

\begin{Cor}
\label{CorO} 
Assume that Newton polygons $\beta\preceq\gamma\preceq\delta$,
a point $x\in V_\beta$, and a generic point $z\in V_{\delta}$
are given such that $x\leq z$.
Then there is a generic point $y\in V_\gamma$ 
such that $x\leq y\leq z$.
\end{Cor}

\begin{proof}
By a change of $\beta$ we may assume that $x$ lies in $V_\beta^\circ$.
Let $Z$ be the irreducible component of $V_\delta$ 
that contains $z$ and let $x'$ be a generic point of
$Z\cap V_\beta$ such that $x\leq x'$. 
Then $x'$ also lies in $V_\beta^\circ$. 
If $n$ denotes the codimension of $x'$ in $Z$, 
the Purity Theorem \cite{deJong-Oort}, Theorem 4.1
implies that $n\leq\codim(\beta)-\codim(\delta)$.
By Theorem \ref{ThO} we must have equality and 
$x'$ is a generic point of $V_\beta$, 
in particular $a(\GGG_{x'})=1$.
By Proposition \ref{PrO} there is a unique generic 
point $y$ of $V_\gamma$ between $x'$ and $z$.
\end{proof}

\section{Reduction to one-dimensional groups}
\label{SecRed}

Let $\bar R=R/pR$ as before.
We begin with the proof of the second part of Theorem \ref{Th1}
in the case where $x$ is the generic point of $\Spec\bar R$,
or equivalently $e=c$.

\begin{proof}[Proof of Theorem \ref{Th1} part ii) if $e=c$]
We may assume that $d\geq 1$.
Let $\beta=\NNN(G)$ and let $\gamma$ be the Newton 
polygon given by the following slope sequence.
$$
\gamma=(\,\underbrace{\textstyle\frac{1}{c+1},\ldots,\frac{1}{c+1}}_{c+1},
\underbrace{1,\ldots,1}_{d-1}\,)
$$
Then $\beta\preceq\gamma$ because $G$ is assumed infinitesimal.
Choose a generic point $\Fp$ of the 
Newton stratum $V_{\gamma}$ of $\Spec\bar R$, 
let $S$ be the completion of the local ring $\bar R_\Fp$,
and let $K$ be its residue field.
Since $\bar R$ is regular and regularity is preserved 
under localisations and completions, 
$S$ is a complete regular local ring.
By Theorem \ref{ThO} the dimension of $S$ is $c$
and the Newton polygon of $\GGG\otimes K$ is $\gamma$.

Let $K'$ be an algebraic closure of $K$ 
and let $S'=K'[[t_1,\ldots,t_c]]$.
The projection $S\to K$ admits a section $K\to S$
in the category of $k$-algebras as $k$ is perfect,
so there is an isomorphism of $k$-algebras
$S\cong K[[t_1,\ldots,t_c]]$
compatible with the projections to $K$. 
We choose one such isomorphism.
Then $S$ becomes a subring of $S'$ by $t_i\mapsto t_i$
and $S'$ becomes an $R$-algebra by the composition
$R\to S\to S'$.

By the choice of $\gamma$, the fibre $\GGG\otimes K'$ has 
a unique $p$-divisible subgroup $M'$ of multiplicative type 
and dimension $d-1$. The quotient $H'=(\GGG\otimes K')/M'$ 
is a one-dimensional infinitesimal $p$-divisible group over $K'$
of height $c+1$.
There is a unique lift of $M'$ to a $p$-divisible subgroup 
of multiplicative type $\MMM'\subset\GGG\otimes S'$, 
see \cite{Groth}, chapitre 3.1 and \cite{deJong}, Lemma 2.4.4,
and the quotient $\HHH'=(\GGG\otimes S')/\MMM'$ 
is a deformation of $H'$ over $S'$.

\begin{Lemma}
\label{Le1}
In this situation $\HHH'$ is a universal deformation of $H'$.
\end{Lemma}

We postpone the proof until section \ref{SecUni}.
If $f:\Spec S'\to\Spec\bar R$ denotes the chosen morphism
and $y\in\Spec S'$ is the generic point 
then $f(y)=x$ because $f$ is flat.
The natural homomorphism 
$\GGG(\bar x)\cong\GGG(\bar y)\to\GGG'(\bar y)$
is bijective since $\MMM(\bar y)$ is the zero group, 
whence the following commutative diagram.
$$
\xymatrix@M+0.2em@R-1em{
{\Gal(\bar y/y)} \ar[r]^-{\rho_y} \ar[d] &
{\GL(T_p\HHH'(\bar y))} \ar@{-}[d]^{\cong} \\
{\Gal(\bar x/x)} \ar[r]^-{\rho_x} &
{\GL(T_p\GGG(\bar x))} }
$$
Here $\rho_y$ is surjective by \cite{Strauch07}, Theorem 2.1,
so $\rho_x$ is surjective as well.
\end{proof}

A modification of the argument gives the 
second part of Theorem \ref{Th1} in general.

\begin{proof}[Proof of Theorem \ref{Th1} part ii)]
By Theorem \ref{ThO} combined with Proposition \ref{PrO}
the generic points of $U_e$ are precisely the generic points
of the Newton stratum $V_\varepsilon$ where $\varepsilon$
is the lowest Newton polygon with exactly $e$ zeros,
$$
\varepsilon=(\,\underbrace{0,\ldots,0}_e,
\underbrace{\textstyle\frac{1}{c-e+1},\ldots,\frac{1}{c-e+1}}_{c-e+1},
\underbrace{1,\ldots,1}_{d-1}\,).
$$
Let $\gamma$, $\Fp$, $S'$, $\HHH'$, 
and $f:\Spec S'\to\Spec\bar R$ be chosen 
exactly as before with the additional requirement
that $\Fp\leq x$. This is possible by Corollary \ref{CorO} 
applied to the points $\Fm_{\bar R}\leq x$.
The inverse image $f^{-1}(U_e)$ is equal 
to the locus $U_e'$ in $\Spec S'$ where 
the etale rank of $\GGG'$ is equal to $e$.
Let $y\in U_e'$ be the unique generic point.
Since $f$ is flat and since $\Fp\leq x$ we have $f(y)=x$.
Now the proof continues as before.
\end{proof}

\section{Galois action in characteristic zero}
\label{SecNull}

Assume now that $x$ is the generic point of $\Spec R$.
Let $\TT=T_p\GGG(\bar x)$ and let 
$$
\chi:\Gal(\bar x/x)\to\ZZ_p^*
$$ 
be the cyclotomic character.
By \cite{Raynaud}, Theorem 4.2.1 and its proof,
$\Gal(\bar x/x)$ acts on $\Lambda^{c+d}(\TT)$ by $\chi^d$.
Let $\GL'(\TT)$ denote the subgroup of all elements of 
$\GL(\TT)$ whose determinant is a $d$-th power and let  
$\Gal^\circ(\bar x/x)$ be the kernel of $\chi$.
The homomorphism $\rho_x$ induces the following  
homomorphisms $\rho'$ and $\rho^\circ$.
$$
\xymatrix@M+0.2em@R-2.8em{
{\rho':\Gal(\bar x/x)} \ar[r] & \GL'(\TT) \\
{\;\;\;\;\;\cup} & \cup \\
{\rho^\circ:\Gal^\circ(\bar x/x)} \ar[r] & \SL(\TT) \\
}
$$

\begin{Lemma}
\label{Le2}
If $\rho^\circ$ is surjective then so is $\rho'$. 
For $d=1$ the converse also holds.
\end{Lemma}

\begin{proof}
We have a homomorphism of exact sequences:
$$
\xymatrix@M+0.2em@R-1em{
1 \ar[r] & 
{\Gal^\circ(\bar x/x)} \ar[r] \ar[d]_{\rho^\circ} &
{\Gal(\bar x/x)} \ar[r]^-{\chi} \ar[d]_{\rho'} &
{\ZZ_p^*} \ar[r] \ar@{->>}[d]^{(\:\:)^d} & 1 \\
1 \ar[r] & 
{\SL(\TT)} \ar[r] &
{\GL'(\TT)} \ar[r]^-{\det} &
{\ZZ_p^{*,d}} \ar[r] & 1
}
$$
Both assertions follow easily.
\end{proof}

\begin{Thm}
\label{Th2}
If $G$ is bi-infinitesimal then $\rho^\circ$ is surjective.
\end{Thm}

In view of Lemma \ref{Le2} 
this is a refinement of Theorem \ref{Th1} i), 
moreover the case $d=1$ follows from
\cite{Rosen-Zimmermann} and \cite{Zimmermann} 
or from \cite{Strauch06}, Theorem 2.1.2. 
By duality this also gives the case $c=1$ because the
Tate module of $\GGG^\vee$ is $\Hom(\TT,\ZZ_p(1))$.

\begin{proof}[Proof of Theorem \ref{Th2}]
We may assume that $c,d\geq 2$.
Let $\beta=\NNN(G)$ and consider the following
Newton polygons $\gamma_i$
with the same endpoints as $\beta$.
\begin{align*}
\gamma_1&{}=(\,\underbrace{\textstyle\frac{1}{c+1},\ldots,\frac{1}{c+1}}_{c+1},
\underbrace{1,\ldots,1}_{d-1}\,) \\
\gamma_2&{}=(\,\underbrace{0,\ldots,0}_{c-1},
\underbrace{\textstyle\frac{d}{d+1},\ldots,\frac{d}{d+1}}_{d+1}\,)
\end{align*}
Since $G$ is bi-infinitesimal we have $\beta\preceq\gamma_i$.
Let $\Fp_i\in\Spec R$ be a generic point of
the Newton stratum $V_{\gamma_i}\subseteq\Spec R/pR$.
By Theorem \ref{ThO} the Newton polygon of the fibre
$\GGG_{\Fp_i}$ is $\gamma_i$ and the codimension of 
$\Fp_i$ in $\Spec R$ is $c_i+1$ where $c_1=c$ and $c_2=d$.
The complete local ring $S_i=\hatR_{\Fp_i}$ 
is regular and unramified in the sense that $p$ 
is part of a minimal set of generators of the
maximal ideal.
Let $S_i'$ be an unramified regular complete local ring 
of dimension $c_i+1$ whose residue field $K_i'$ is an 
algebraic closure of the residue field $K_i$ of $S_i$
and choose an embedding $S_i\to S_i'$
such that $S_i'\otimes_{S_i}K_i=K_i'$.

(More explicitly, put $S_i'=W(K_i')[[t_1,\ldots,t_{c_i}]]$;
then choose a Cohen ring $C_i$ in $S_i$, 
an isomorphism of $C_i$-algebras 
$S_i\cong C_i[[t_1,\ldots,t_{c_i}]]$,
and an embedding of $C_i$ into $W(K_i')$; 
extend this to an embedding
$S_i\to S_i'$ by $t_i\mapsto t_i$.)

Let $\Fq\subset R$ and $\Fq_i\subset S_i'$ 
be the prime ideals generated by $p$. 
The complete local rings $A=\hatR_{\Fq}$ 
and $B_i=(\hatS_i')_{\Fq_i}$ are unramified
discrete valuation rings that fit into
a commutative diagram of rings:
$$
\xymatrix@M+0.2em@R-1em{
S_1' \ar[d] & R \ar[r] \ar[d] \ar[l] & S_2' \ar[d] \\
B_1 & A \ar[r] \ar[l] & B_2
}
$$

The scalar extensions of $\GGG$ to these rings 
admit natural filtrations of different types:
since the fibre $\GGG_\Fq$ is ordinary, over $A$ 
there is an exact sequence of $p$-divisible groups
\begin{equation}
\label{EqS1}
0\longrightarrow\MMM\longrightarrow\GGG\otimes A
\longrightarrow\EEE\longrightarrow 0
\end{equation} 
where $\MMM$ is of multiplicative type of height $d$
and $\EEE$ is etale of height $c$. 
By the choice of the polygons $\gamma_i$, 
over $S_1'$ there is an exact sequence
\begin{equation}
\label{EqS2}
0\longrightarrow\MMM'\longrightarrow\GGG\otimes S_1'
\longrightarrow\HHH_1\longrightarrow 0
\end{equation} 
where $\MMM'$ is isomorphic to $\mu_{p^\infty}^{d-1}$ 
and $\HHH_1$ is bi-infinitesimal of dimension $1$ and
height $c+1$, while over $S_2'$ there is an exact sequence 
\begin{equation}
\label{EqS3}
0\longrightarrow\HHH_2\longrightarrow\GGG\otimes S_2'
\longrightarrow\EEE'\longrightarrow 0
\end{equation} 
where $\EEE'$ is isomorphic to $(\QQ_p/\ZZ_p)^{c-1}$ 
and $\HHH_2$ is bi-infinitesimal of dimension $d$ 
and height $d+1$.
In both cases $\HHH_i$ is the universal deformation
of its special fibre over $W(K_i)$-algebras because
$\HHH_i\otimes S_i/pS_i$ is the universal deformation 
over $K_i$-algebras according to Lemma \ref{Le1} 
(applied to the dual if $i=2$). 
Since over $B_i$ all homomorphisms from groups of 
multiplicative type to etale groups are trivial, 
as subgroups of $\GGG\otimes B_1$ and $\GGG\otimes B_2$ 
we have 
\begin{equation}
\label{EqS4}
\MMM'\otimes_{S_1'}B_1\subseteq\MMM\otimes_AB_1,\qquad
\MMM\otimes_AB_2\subseteq\HHH_2\otimes_{S_2'}B_2.
\end{equation}

Let $F'$ be an algebraic closure of the residue field of $A$, 
let $A'=W(F')$, and choose an embedding $\sigma:A\to A'$ 
extending the given homomorphism $A\to F'$. 
This time the choice makes a difference 
and will be fixed later.
In order to relate the various Galois actions on the Tate 
module we choose an algebraically closed field $\Omega$
together with embeddings of $A'$ and both $B_i$ into 
$\Omega$ that coincide over $A$. 
For every subring $X$ of $\Omega$ 
let $\Gal_X=\pi_1(\Quot(X),\Omega)$
and denote by $\Gal^\circ_X$ the kernel 
of the cyclotomic character $\Gal_X\to\ZZ_p^*$.

If we write $\TT=T_p\GGG(\Omega)$ by a harmless change
of notation, we have to show that the natural homomorphism
$\rho^\circ_R:\Gal^\circ_R\to\SL(\TT)$ is surjective. Let
\begin{gather*}
\TT_1=T_p\HHH_1(\Omega),\qquad\hspace{0.25em}
\EE=T_p\EEE(\Omega),\qquad\hspace{0.9em}
\EE'=T_p\EEE'(\Omega),\hspace{0.55em}\\
\TT_2=T_p\HHH_2(\Omega),\qquad
\MM=T_p\MMM(\Omega),\qquad
\MM'=T_p\MMM'(\Omega).
\end{gather*}
From \eqref{EqS2}, \eqref{EqS1}, \eqref{EqS3} 
in that order we obtain
the following exact sequences of free $\ZZ_p$-modules
with actions of the designated groups $\Gal_X$
where the action on $\TT$ is induced from the action
of $\Gal_R$ by the natural homomorphism $\Gal_X\to\Gal_R$.
The vertical arrows exist by \eqref{EqS4}.
$$
\xymatrix@M+0.2em@R-1em{
\Gal_{S_1'} &
0 \ar[r] &
\MM' \ar[r] \ar[d] &
\TT \ar[r] \ar@{=}[d] &
\TT_1 \ar[d] \ar[r] & 0 \\
\Gal_{A'} &
0 \ar[r] & 
\MM \ar[r] \ar[d] &
\TT \ar[r] \ar@{=}[d] &
\EE \ar[r] \ar[d] & 0 \\
\Gal_{S_2'} &
0 \ar[r] &
\TT_2 \ar[r] &
\TT \ar[r] &
\EE' \ar[r] & 0
}
$$

Here $\Gal^\circ_{S_1'}$ acts trivially on $\MM'$
and $\Gal_{S_2'}$ acts trivially $\EE'$.
By the known cases $d=1$ and $c=1$ of Theorem \ref{Th2},
the induced homomorphisms $\Gal^\circ_{S_i'}\to\SL(\TT_i)$ 
are surjective.
In many cases this already implies that $\rho^\circ_R$ 
is surjective, but in order to conclude in general
we also need the action of $\Gal^\circ_{A'}$.
Let $U\subseteq\SL(\TT)$ be the unipotent subgroup
that acts trivially on $\MM$ and on $\EE$, thus
$U\cong\ZZ_p^{cd}$. Then $\Gal^\circ_{A'}$ acts on
$\TT$ by a homomorphism
$$
\rho^\circ_{A'}:\Gal^\circ_{A'}\to U.
$$

\begin{Lemma}
\label{Le3}
For a suitable choice of the embedding $\sigma:A\to A'$, 
the homomorphism $\rho^\circ_{A'}$ is surjective.
\end{Lemma}

We postpone the proof of Lemma \ref{Le3} until section \ref{SecGen}
and continue in the proof of Theorem \ref{Th2}.
Let $U_1$ denote the group of all elements of $\SL(\TT)$
that act trivially on $\MM'$ and on $\TT_1$, 
let $U_2\subseteq\SL(\TT)$ be the group that acts trivially
on $\TT_2$ and on $\EE'$, 
and let $H$ be the image of $\Gal^\circ_R\to\SL(\TT)$.
Then $H$ contains $U$ by Lemma \ref{Le3}, 
so $H\cap U_1$ contains $U\cap U_1$.
Since $\Gal^\circ_{S_1'}\to\SL(\TT_1)$ is surjective,
$H\cap U_1$ is invariant under the conjugation
action of $\SL(\TT_1)$ on $U_1$.
Thus $H\cap U_1=U_1$ and similarly $H\cap U_2=U_2$.
It follows that $H$ contains the (pointwise) stabilisers
of $\MM'\subset\TT$ and of
$\EE^{\prime\vee}\subset\TT^\vee$. 
These generate $\SL(\TT)$ as is easily shown
by straightforward considerations of matrices.
\end{proof}

The following example shows that the use 
of Lemma \ref{Le3} cannot be avoided.

\begin{Example}
There is a subgroup $H$ of $\GL_4(\FF_2)$
of index $8$ with the following property:
If $P_1$ and $P_2$ denote the standard parabolic 
subgroups of $\GL_4$ of types $(1,3)$ and $(3,1)$ 
then the natural projections $\pi_i:H\cap P_i\to\GL_3(\FF_2)$
are surjective. The group $H$ is generated 
by the following matrices $A_1,\ldots, A_6$.
\begin{gather*}
\left(
\begin{smallmatrix}
1 & 0 & 0 & 0 \\
1 & 1 & 0 & 1 \\
0 & 0 & 1 & 1 \\
0 & 0 & 0 & 1 
\end{smallmatrix}
\right)\;\;
\left(
\begin{smallmatrix}
1 & 0 & 0 & 1 \\
0 & 1 & 0 & 0 \\
0 & 1 & 1 & 1 \\
0 & 0 & 0 & 1 
\end{smallmatrix}
\right)\;\;
\left(
\begin{smallmatrix}
1 & 1 & 0 & 0 \\
0 & 1 & 0 & 0 \\
0 & 0 & 1 & 1 \\
0 & 0 & 0 & 1 
\end{smallmatrix}
\right)\;\;
\left(
\begin{smallmatrix}
1 & 0 & 0 & 1 \\
0 & 1 & 1 & 0 \\
0 & 0 & 1 & 0 \\
0 & 0 & 0 & 1 
\end{smallmatrix}
\right)\;\;
\left(
\begin{smallmatrix}
1 & 1 & 1 & 0 \\
0 & 1 & 0 & 0 \\
0 & 0 & 1 & 0 \\
0 & 0 & 1 & 1 
\end{smallmatrix}
\right)\;\;
\left(
\begin{smallmatrix}
1 & 1 & 0 & 1 \\
0 & 1 & 0 & 0 \\
0 & 1 & 1 & 0 \\
0 & 0 & 0 & 1 
\end{smallmatrix}
\right)
\end{gather*}
It is easy to see that $\pi_2(A_1),\ldots,\pi_2(A_4)$ 
generate $\GL_3(\FF_2)$ and that $\pi_1(A_3),\ldots,\pi_1(A_6)$
also generate this group.
The index of $H$ in $\GL_4(\FF_2)$
has been computed only electronically.
It would be interesting to have a definition of 
$H$ without generators.
\end{Example}

\section{Deformations of $p$-divisible groups}
\label{SecDef}

Before proving Lemmas \ref{Le1} and \ref{Le3}
let us recall some aspects of the
deformation theory of $p$-divisible groups.
Let $G$ be a $p$-divisible group over an arbitrary 
ring $R$ in which $p$ is nilpotent
and write $\Lambda_G=\Hom_R(\Lie G,\omega_{G^\vee})$.
If $R=S/I$ where $S$ is an $I$-adically complete ring
let $\Def_{S/R}(G)$ denote the set of isomorphism
classes of lifts of $G$ to $S$.

\begin{Thm}
\label{ThD}
If $I^2=0$ then $\Def_{S/R}(G)$ is naturally 
a torsor under the $R$-module
$$
\Hom_R(\omega_{G^\vee},I\otimes\Lie G)=
\Hom_R(\Lambda_G,I).
$$ 
\end{Thm}

This is classical and follows either from 
\cite{Illusie}, th\'{e}or\`{e}me 4.4 and corollaire 4.7
or (except for the existence of lifts) from the crystalline
deformation theorem \cite{Messing}, V, Theorem 1.6,
because the set of lifts to $\DD(G)_S$ of the Hodge filtration
$$
0\longrightarrow\omega_{G^\vee}\overset i\longrightarrow\DD(G)_R
\overset\pi\longrightarrow\Lie G\longrightarrow 0
$$
is a torsor under $\Hom_R(\omega_{G^\vee},I\otimes\Lie G)$;
here $\DD(G)$ is the covariant Dieudonn\'{e} crystal
defined in loc.\ cit.
Both constructions give the same action of 
$\Hom_R(\Lambda_G,I)$ on $\Def_{S/R}(G)$
but we could not find a reference for this fact
(and will not use it).
As a formal consequence of Theorem \ref{ThD}, 
for every $p$-divisible group $G$ over $R$ as above 
there is a `Kodaira-Spencer' homomorphism  
$\kappa_G':\omega_{G^\vee}\to\Omega_R\otimes\Lie G$,
or equivalently
$$
\kappa_G:\Lambda_G\to\Omega_R,
$$
uniquely determined by the following property.
For any ring homomorphism $f:R\to A$ where 
$A=B/I$ such that $I^2=0$ denote by $\Lift_{B/A}(f)$ 
the set of ring homomorphisms $R\to B$ lifting $f$,
which is either the empty set or a torsor under
the $A$-module $\Hom_R(\Omega_R,I)$. 
Then the obvious map
\begin{equation}
\label{EqD1}
\Lift_{B/A}(f)\to\Def_{B/A}(G\otimes_RA)
\end{equation}
is equivariant with respect to the homomorphism
$\Hom_R(\Omega_R,I)\to\Hom_R(\Lambda_G,I)$
induced by $\kappa_G$. The homomorphism $\kappa_G$ 
is functorial in $R$ in the obvious sense.
If one uses the crystalline construction of the
torsor structure in Theorem \ref{ThD} 
then $\kappa'_G$ can be written down directly in terms of 
the connection $\nabla:\DD(G)_R\to\Omega_R\otimes\DD(G)_R$,
namely $\kappa'_G=(\id\otimes\pi)\circ\nabla\circ i$.

A homomorphism of $p$-divisible groups $G\to H$ over $R$
induces a morphism of arrows (a commutative square)
$\kappa'_G\to\kappa'_H$.
In the special case of an exact sequence 
of $p$-divisible groups $0\to M\to G\to H\to 0$ 
where $M$ is of multiplicative type, 
thus $\omega_{G^\vee}\cong\omega_{H^\vee}$, 
this translates into a commutative triangle 
with split injective $\lambda$ :
\begin{equation}
\label{EqD2}
\xymatrix@M+0.2em@R-1em@C-1em{
{\Lambda_H} \ar@{^(->}[rr]^{\lambda} \ar[dr]_{\kappa_H} &&
{\Lambda_G} \ar[dl]^{\kappa_G} \\ 
& {\Omega_R} 
}
\end{equation} 

If $G$ is a $p$-divisible group over a field $k$ 
of characteristic $p$ and $\GGG$ is a deformation of $G$ 
over a complete local noetherian $k$-algebra $R$ with
residue field $k$ we may consider the following 
composite homomorphism $\bar\kappa_\GGG$.
$$
\bar\kappa_\GGG:
\Lambda_G\cong\Lambda_\GGG\otimes k
\xrightarrow{\kappa_\GGG\otimes\id}
\Omega_R\otimes k\to\Omega_{R/k}\otimes k
\cong\Fm_R/\Fm_R^2
$$
The deformation $\GGG$ is universal if and only if
$R$ is regular and $\bar\kappa_\GGG$ is bijective;
let us call $\GGG$ versal if $R$ is regular and 
$\bar\kappa_\GGG$ is injective.
In the universal case $\kappa_\GGG$ induces
an isomorphism $\Lambda_\GGG\cong\hatOmega_{R/k}$
because both modules are free over $R$.
If $\GGG$ is universal and $k$ is perfect then
$\kappa_\GGG$ is an isomorphism $\Lambda_\GGG\cong\Omega_R$ 
because in that case $\Omega_R\cong\hatOmega_{R/k}$.

As announced earlier let us conclude this section
with a proof of Proposition \ref{PrO}.

\begin{proof}[Proof of Proposition \ref{PrO}]
We use and recall the transitivity argument 
of \cite{Oort-PM}, Proposition 2.8.
Let $K$ be the residue field of $\bar R$ at $x$,
let $\tilde R=\bar R[[t_1,\ldots,t_{cd}]]$
and $T=K[[t_1,\ldots,t_{cd}]]$.
By \cite{deJong} 2.2.2 and 2.2.4
or by the theory of displays
there is a deformation $\tilde\GGG$ over $\tilde R$
of $\GGG\otimes\bar R$ so that 
$\tilde\GGG_T=\tilde\GGG\otimes_{\tilde R}T$
is the universal deformation in equal characteristic 
of its special fibre $\GGG\otimes_RK$. 
Universality of $\GGG$ gives a homomorphism 
$\varphi:\bar R\to\tilde R$ such that
$\tilde\GGG\cong\GGG\otimes_\varphi\tilde R$
as deformations of $G$.
Since $\bar R\to \tilde R\to\bar R$ is the identity,
there is a unique local homomorphism $\psi:\bar R_x\to T$
making the following commute.
$$
\xymatrix@M+0.2em@R-1em{
\bar R \ar[r]^-\varphi \ar[d] &
\tilde R=\bar R[[t_1,\ldots,t_{cd}]] \ar[d] \\
\bar R_x \ar[r]^-\psi &
T=K[[t_1,\ldots,t_{cd}]]
}
$$
It follows that the universal group $\tilde\GGG_T$
is isomorphic to $(\GGG\otimes\bar R_x)\otimes_\psi T$;
in particular, the inverse images of the various $V_{\beta,x}$ 
under $\Spec\psi:\Spec T\to\Spec\bar R_x$ form the
Newton stratification given by $\tilde\GGG_T$.
Since that stratification has the required properties 
by \cite{Oort-Ann}, Theorem 3.2, the proposition follows
if we show that $T$ is isomorphic via $\psi$ to a 
power series ring over the completion of $\bar R_x$.

Let $\Fm\subset\bar R_x$ and $\Fn\subset T$ be 
the maximal ideals. Since $\bar R_x$ is regular 
and $\psi$ induces an isomorphism on residue fields,
it suffices that $\Fm/\Fm^2\to\Fn/\Fn^2$ is injective.
Now $\Fm/\Fm^2$ is a submodule of $\Omega_{\bar R_x}\otimes K$ 
and $\Fn/\Fn^2$ is isomorphic to $\hatOmega_{T/K}\otimes K$.
Hence it is enough to show that  
$\Omega_{\bar R_x}\otimes_\psi T\to\hatOmega_{T/K}$
is an isomorphism. This follows because clearly
$\Lambda_{\GGG\otimes\bar R_x}\otimes_\psi T\to\Lambda_{\tilde\GGG_T}$
is an isomorphism and the Kodaira-Spencer homomorphisms 
$\Lambda_{\GGG\otimes\bar R_x}\to\Omega_{\bar R_x}$
and $\Lambda_{\tilde\GGG_T}\to\hatOmega_{T/K}$ 
are isomorphisms by universality of $\GGG$ and of $\tilde\GGG_T$.
\end{proof}

\section{Universality over completions}
\label{SecUni}

This section is devoted to the proof of Lemma \ref{Le1},
but let us consider a more general situation.
Assume that $G$ is a $p$-divisible group
over a perfect field $k$ of characteristic $p$, let $R$ 
be its universal deformation ring over $k$, thus 
$R\cong k[[t_1,\ldots,t_{cd}]]$, and denote by
$\GGG$ the universal deformation over $R$.

For an arbitrary prime $\Fp\in\Spec R$
we consider the complete local ring $S=\hatR_\Fp$ 
with residue field $K=S/\Fm_S$.
The maximal subgroup of multiplicative type 
$M$ of $\GGG\otimes K$ lifts uniquely to a subgroup 
of multiplicative type $\MMM$ of $\GGG\otimes S$. 
The quotient $\HHH=(\GGG\otimes S)/\MMM$ 
is a deformation over $S$ of the $p$-divisible group 
$H=(\GGG\otimes K)/M$ over $K$. 
To ask whether $\HHH$ is a universal or versal 
deformation of $H$ makes sense only after a structure
of $K$-algebra is chosen on $S$, and in general 
the answer depends on the choice 
(but not in the special case of Lemma \ref{Le1}).

Denote by $\Sigma$ (or $\bar\Sigma$) the set of all 
$k$-algebra homomorphisms $\sigma:K\to S$ 
(or $\bar\sigma:K\to S/\Fm_S^2$) 
lifting the identity of $K$.
Since $k$ is perfect, $\Sigma$ is non-empty 
and the reduction map $\Sigma\to\bar\Sigma$ is surjective. 
The set $\bar\Sigma$ is a torsor under the 
finite dimensional $K$-vector space 
$\Hom_K(\Omega_K,\Fm_S/\Fm_S^2)$, 
so the Zariski topology on the vector space 
induces a well-defined topology on $\bar\Sigma$.

\begin{Prop}
\label{Pr1}
There is an open subset $U$ of\/ $\bar\Sigma$ such that 
$\HHH$ is versal with respect to some $\sigma\in\Sigma$ 
if and only if its reduction $\bar\sigma$ lies in $U$.  
The set $U$ is non-empty if and only if 
$\dim(S)\geq\dim_K(\Lambda_H)$.
We have $U=\bar\Sigma$ if and only if  
$\kappa_H:\Lambda_H\to\Omega_K$ is zero.
\end{Prop}

Note that if $\dim(S)=\dim_K(\Lambda_H)$ then 
`versal' is equivalent to `universal'.

\begin{Lemma}
\label{Le4}
The natural homomorphism 
$\Omega_R\otimes_RS\to\Omega_S$ 
is an isomorphism.
\end{Lemma}

\begin{proof}
Since this is true with $R_\Fp$ in place of $S$, it suffices 
that $\Omega_{R_\Fp}\otimes_{R_\Fp} S\to\Omega_S$ is an isomorphism. 
The ring $S$ has a finite $p$-basis because
it is isomorphic to a power series ring over the field $K$
which has a finite $p$-basis because this holds for $R$.
Thus $\Omega_S$ is a finite $S$-module, whence
$$
\Omega_S\cong\varprojlim_n(\Omega_S/\Fm_S^n\Omega_S)
\cong\varprojlim_n\Omega_{S/\Fm_S^n}.
$$
Using that $\Omega_{R_\Fp}$ is a finite $R_\Fp$-module, 
the same reasoning shows that $\Omega_{R_\Fp}\otimes_{R_\Fp} S$
is isomorphic to $\varprojlim\Omega_{S/\Fm_S^n}$ as well.
\end{proof}

\begin{proof}[Proof of Proposition \ref{Pr1}]
Assume that $\sigma:K\to S$ is given.
Since $S$ is regular, $\HHH$ is versal with respect to
$\sigma$ if and only if the homomorphism $\bar\kappa_\HHH$ 
defined by the upper triangle of the following commutative
diagram is injective. The lower triangle is \eqref{EqD2}
and is independent of $\sigma$. The homomorphism 
$\kappa_{\GGG\otimes S}:\Lambda_{\GGG\otimes S}\to\Omega_S$ 
is an isomorphism because it can be identified with 
$\kappa_\GGG\otimes\id:\Lambda_\GGG\otimes_R S\to\Omega_R\otimes_R S$
by Lemma \ref{Le4} and because $\kappa_\GGG$ 
is an isomorphism as $\GGG$ is universal and $k$ is perfect.
$$
\xymatrix@M+0.2em@C+1em{
\Lambda_\HHH\otimes_SK \ar[r]^-{\bar\kappa_\HHH} 
\ar@{_(->}[d]_\lambda \ar[dr]^{\kappa_\HHH\otimes\id} & 
\Omega_{S/K}\otimes_SK \\
\Lambda_{\GGG\otimes S}\otimes_S K 
\ar[r]_-{\kappa_{\GGG\otimes S}\otimes\id}^\cong &
\Omega_S\otimes_SK \ar[u]_v
}
$$
In order to see how the kernel of $v$ varies with $\sigma$
we write down two standard exact sequences for 
modules of differentials.
The first depends on $\sigma$, the second does not.
\begin{gather*}
0\longrightarrow\Omega_K\overset u\longrightarrow\Omega_S\otimes_S K
\overset v\longrightarrow\Omega_{S/K}\otimes_S K\longrightarrow 0 \\
0\longrightarrow\Fm_S/\Fm_S^2\overset d\longrightarrow\Omega_S\otimes_S K
\overset\pi\longrightarrow\Omega_K\longrightarrow 0
\end{gather*}
Clearly $v\circ d$ is an isomorphism and
$\pi\circ u=\id$, which proves exactness on the left.
If $\sigma$ is changed so that $\bar\sigma$ is altered
by $\delta:\Omega_K\to\Fm_S/\Fm_S^2$ then $u$
changes by $d\circ\delta$. 
It follows that $\bar\Sigma$ is bijective to the set of
homomorphisms $u$ with $\pi\circ u=\id$, 
which reduces the proposition to linear algebra.
Namely, $\bar\kappa_\HHH$ is injective if and only if
the images of $u$ and of $\kappa_\HHH\otimes\id$
in $\Omega_S\otimes K$ have zero intersection.
This condition defines an open subset $U$ of $\bar\Sigma$ 
that is non-empty if and only if
$\dim_K(\Lambda_H)\leq\dim_K(\Omega_{S/K}\otimes K)=\dim(S)$.
The intersection is zero for all choices of $u$ if and only if 
the composition $\pi\circ(\kappa_\HHH\otimes\id)$ is zero,
but this composition is just $\kappa_H$.
\end{proof}

\begin{proof}[Proof of Lemma \ref{Le1}]
Clearly $\HHH'$ is universal if and only if 
the group $\HHH$ considered above is universal 
with respect to the chosen section $\sigma:K\to S$. 
Since $\dim_K(\Lambda_H)=c=\dim(S)$,
Proposition \ref{Pr1} implies that $\HHH$ is universal if 
$\bar\sigma$ lies in a dense open subset of $\bar\Sigma$,
which is sufficient for our applications.
In order that $\HHH$ is universal for all choices of $\sigma$,
we need in addition that $\kappa_H$ vanishes.
Since $H$ is a one-dimensional formal group, 
its base change to a separable closure $K^{\sep}$
is defined over the prime field $\FF_p$
by \cite{Zink}, Satz 5.33.
Using that $\Omega_{K^{\sep}}\cong\Omega_K\otimes_KK^{\sep}$
and $\Omega_{\FF_p}=0$, the vanishing of 
$\kappa_H$ follows by its functoriality with respect 
to the base ring.
\end{proof}

An alternative proof of Lemma \ref{Le1} in the case where $a(G)=1$
is given in \cite{Tian08} using the linear structure of the
deformation space in that case.

\section{Generic completion of the universal deformation}
\label{SecGen}

In this section we prove Lemma \ref{Le3}.
As in the previous section let $G$ be a $p$-divisible 
group over a perfect field $k$ of characteristic $p$,
but now let again $R$ be the universal deformation ring 
of $G$ over $W(k)$ and $\GGG$ the universal deformation over $R$.

Let us consider the unramified discrete valuation rings
$A=\hatR_{(p)}$ and $A'=W(F')$ where $F'$ is a 
fixed algebraic closure of the residue field $F$ of $A$.
Let $\Sigma$ be the set of ring homomorphisms 
$A\to A'$ that induce the given embedding $F\to F'$ 
modulo $p$.

\begin{Prop}
\label{Pr2}
The map 
$\psi:\Sigma\to\Def_{A'/F'}(\GGG\otimes F')$
that maps a homomorphism $\sigma:A\to A'$ to the scalar 
extension of $\GGG\otimes A$ by $\sigma$ is bijective.
\end{Prop}

\begin{proof}
Let $\bar R=R/pR$.
Since $\GGG\otimes\bar R$ is the universal deformation of $G$ 
in characteristic $p$, the homomorphism $\kappa_{\GGG\otimes\bar R}$
is an isomorphism, whence an isomorphism
$$
\kappa_{\GGG\otimes F}:\Lambda_{\GGG\otimes F}\cong\Omega_F
$$
as $F$ is the quotient field of $\bar R$. 
The proposition is a formal consequence:
Let $A'_n=A'/p^nA'$. It suffices that for every $n\geq 1$ and 
every homomorphism $\sigma:A\to A'_n$ lifting $F\to F'$,
using the notation of \eqref{EqD1} the obvious map
$$
\Lift_{A'_{n+1}/A'_n}(\sigma)\to
\Def_{A'_{n+1}/A'_n}(\sigma_*(\GGG\otimes A))
$$
is bijective. Since its source is non-empty by
\cite{Berthelot-Messing}, Proposition 1.2.6 and 
since $\Omega_A/p\Omega_A=\Omega_F$, this is an 
equivariant map of torsors with respect to the 
homomorphism of $F'$-vector spaces
$$
\Hom_F(\Omega_F,p^nA'/p^{n+1}A')\to\Hom_F(\Lambda_{\GGG_F},p^nA'/p^{n+1}A')
$$
induced by $\kappa_{\GGG\otimes F}$, which is bijective.
\end{proof}

In order to deduce Lemma \ref{Le3} we have to describe 
the Galois representation on the Tate module of an 
arbitrary deformation over $A'$ of $H=\GGG\otimes F'$.
Let $K$ be the quotient field of $A'$, choose an
algebraically closed field $\Omega$ containing $K$,
and write $\Gal_{K}=\pi_1(K,\Omega)$. Also fix
an isomorphism of $p$-divisible groups over $F'$ 
$$
H\cong(\QQ_p/\ZZ_p)^c\oplus\mu_{p^\infty}^d.
$$

\begin{Lemma}
\label{Le5}
The map of sets $e:\Def_{A'/F'}(H)\to H^1(\Gal_{K},\ZZ_p(1))^{cd}$
that maps a deformation $\HHH$ to the isomorphism class 
of the associated extension of\/ $\Gal_K$-modules
$$
0\longrightarrow\ZZ_p(1)^d\longrightarrow T_p\HHH(\Omega)
\longrightarrow\ZZ_p^c\longrightarrow 0
$$
can be written as a composition
$$
\Def_{A'/F'}(H)
\xrightarrow{\alpha}
\hatGm(A')^{cd}
\xrightarrow{i}
(K^*)^{cd}
\xrightarrow{\delta}
H^1(\Gal_{K},\ZZ_p(1))^{cd}
$$
where $\alpha$ is bijective, $i$ is the natural inclusion,
and $\delta$ is the Kummer homomorphism.
\end{Lemma}

\begin{proof}
We have a bijection
$\gamma:\Def_{A'/F'}(H)\cong\Ext^1_{A'}(\QQ_p/\ZZ_p,\mu_{p^\infty})^{cd}$
and an isomorphism 
$\beta:\hatGm(A')\cong\Ext^1_{A'}(\QQ_p/\ZZ_p,\mu_{p^\infty})$
defined as the projective limit of
$$
\beta_n:\mu_{p^\infty}(A'_n)\cong\Ext^1_{A'_n}(\QQ_p/\ZZ_p,\mu_{p^\infty})
$$
where $\beta_n$ is the connecting homomorphism of
$0\to\ZZ\to\ZZ[1/p]\to\QQ_p/\ZZ_p\to 0$. 
These are isomorphisms because 
$\Ext^i_{A'_n}(\ZZ[1/p],\mu_{p^\infty})$ vanishes for $i\leq 1$.
Put $\alpha=(-\beta^{-1})^{cd}\circ\gamma$.
Let $\bar K$ be the algebraic closure of $K$ in $\Omega$.
The required relation $e=\delta\circ i\circ\alpha$ translates 
into anti-commutativity of the following diagram, 
where $\delta_1$ is induced by the exact sequence
$0\to\mu_{p^n}\to\bar K^*\to\bar K^*\to 0$ and $\delta_2$ 
is induced by $0\to\ZZ\to\ZZ\to\ZZ/p^n\ZZ\to 0$, while
$\varepsilon_1$ maps an extension $E$ to $E/p^nE$
and $\varepsilon_2$ maps $E$ to $E[p^n]$.
$$
\xymatrix@M+0.2em{
K^*=\Hom_{\Gal_{K}}(\ZZ,\bar K^*) \ar[r]^-{\delta_1} \ar[d]_{\delta_2} &
\Ext^1_{\Gal_{K}}(\ZZ,\mu_{p^n}) \ar[d]^{\varepsilon_1} \\
\Ext^1_{\Gal_{K}}(\ZZ/p^n\ZZ,\bar K^*) \ar[r]^{\varepsilon_2} &
\Ext^1_{\Gal_{K}}(\ZZ/p^n\ZZ,\mu_{p^n})
}
$$
This is easily checked.
\end{proof}

\begin{proof}[Proof of Lemma \ref{Le3}]
Let $K_{\infty}=K[\mu_{p^\infty}]$. 
The homomorphism $\rho^\circ_{A'}:\Gal^\circ_{A'}\to\ZZ_p^{cd}$
is surjective if and only if its reduction
$\bar\rho^\circ_{A'}:\Gal^\circ_{A'}\to\FF_p^{cd}$ is surjective.
By Proposition \ref{Pr2} and Lemma \ref{Le5}
we have a bijection $a:\Sigma\cong\hatGm(A')^{cd}$ so that 
$\bar\rho^\circ_{A'}$ is the image of $a(\sigma)$ under 
$$
\hatGm(A')\xrightarrow{i}K_\infty^*
\xrightarrow{pr}K_\infty^*/(K_\infty^*)^p
\overset\delta\cong\Hom(\Gal^\circ_{A'},\FF_p)
$$
(componentwise)
where $\delta$ is induced by the Kummer sequence.
Thus $\rho^\circ_{A'}$ is surjective if and only if the
components of $a(\sigma)$ map to linearly independent
elements in the $\FF_p$-vector space $K_\infty^*/(K_\infty^*)^p$.
Since $a(\sigma)$ is arbitrary it suffices
to show that the image of $\hatGm(A')$ in 
$K_\infty^*/(K_\infty^*)^p$ is infinite.

Let $K_n=K[\mu_{p^n}]$ and $V_n=K_n^*/(K_n^*)^p$, 
which for $n\geq 1$ is identified with the group 
$\Hom(\Gal_{K_n},\FF_p)$.
Using that $K_{\infty}$ over $K_2$ is a $\ZZ_p$-extension
it is easy to see that the kernel of $V_2\to V_n$
is independent of $n$ for $n\geq 3$.
Since $V_1\to V_3$ has finite kernel,
it suffices that the image of $\hatGm(A')$ 
in $V_1$ is infinite. 
A consideration of valuations shows that
for every $x\in A'$ of valuation $1$, 
the element $1+x$ does not lie in $(K_1^*)^p$.
In other words, the kernel of $\hatGm(A')\to V_1$ 
is contained in the kernel of the surjection 
$\hatGm(A')\to F'$ given by $1+px\mapsto \bar x$. 
Since $F'$ is infinite the assertion follows.
\end{proof}

\end{document}